\documentclass[12pt]{amsart}

\usepackage{amssymb,amsmath,amsthm, url, comment}

\begin{document}

\newtheorem{thm}{Theorem}
\newtheorem{prop}[thm]{Proposition}
\newtheorem{lem}[thm]{Lemma}
\newtheorem{cor}[thm]{Corollary}

\theoremstyle{definition}
\newtheorem{defn}[thm]{Definition}
\newtheorem{remark}[thm]{Remark}
\newtheorem{remarks}[thm]{Remarks}
\newtheorem{example}[thm]{Example}
\newtheorem{question}[thm]{Question}
\newtheorem{reduction}[thm]{Reduction}
\newtheorem{conjecture}[thm]{Conjecture}
\newtheorem*{ack}{Acknowledgment}

\renewcommand{\labelenumi}{(\alph{enumi})}

\newcommand{\onto}{\twoheadrightarrow}
\newcommand{\mono}{\rightarrowtail}
\newcommand{\into}{\hookrightarrow}
\newcommand{\lra}{\longrightarrow}
\newcommand{\isomoto}{\overset{\sim}{\tto}}
\newcommand{\Mat}{{\operatorname{M}}}
\newcommand{\Mn}{\Mat_n}
\newcommand{\PGLn}{{\operatorname{PGL}}_n}
\newcommand{\PGL}{{\operatorname{PGL}}}
\newcommand{\Sym}{\operatorname{S}}

\newcommand{\cG}{\mathcal{G}}
\newcommand{\cH}{\mathcal{H}}
\newcommand{\cN}{\mathcal{N}}
\newcommand{\cR}{\mathcal{R}}
\newcommand{\cF}{\mathcal{F}}
\newcommand{\cD}{\mathcal{D}}
\newcommand{\bt}{\mathbf{t}}

\newcommand{\Char}{\operatorname{char}}

\newcommand{\bbZ}{\mathbb{Z}}
\newcommand{\bbG}{\mathbb{G}}
\newcommand{\bbC}{\mathbb{C}}
\newcommand{\bbQ}{\mathbb{Q}}
\newcommand{\bbF}{\mathbb{F}}
\newcommand{\bbP}{\mathbb{P}}
\newcommand{\bbA}{\mathbb{A}}
\newcommand{\bbN}{\mathbb{N}}
\newcommand{\bbH}{\mathbb{H}}
\newcommand{\bbR}{\mathbb{R}}
\newcommand{\Rad}{\operatorname{Rad}}

\title[A non-commutative Nullstellensatz]{A non-commutative Nullstellensatz}

\author{Zhengheng Bao}
\address{Department of Mathematics, University of British Columbia,
	Vancouver, BC, Canada, V6T 1Z2}
\email{zhengheng.bao@alumni.ubc.ca}
\thanks{Zhengheng Bao was partially supported by a Work Learn International Undergraduate Research Award (WLIURA) 
at the University of British Columbia} 

\author{Zinovy Reichstein}
 \email{reichst@math.ubc.ca}

\thanks{Zinovy Reichstein was partially supported by
 National Sciences and Engineering Research Council of
 Canada Discovery grant 253424-2017.}


\begin{abstract} Let $K$ be a field and $D$ be a finite-dimensional central division algebra 
over $K$. We prove a variant of the Nullstellensatz for $2$-sided ideals in the ring of
polynomial maps $D^n \to D$. In the case where $D = K$ is commutative, 
our main result reduces to the $K$-Nullstellensatz of Laksov and Adkins-Gianni-Tognoli. In the case, where
$K = \mathbb R$ is the field of real numbers and $D$ is the algebra of Hamilton quaternions, it reduces to 
the quaternionic Nullstellensatz recently proved by Alon and Paran.
\end{abstract}

\maketitle

\section{Introduction}

Let $K$ be a field and $P_{K, n} = K[x_1, \ldots, x_n]$ be the commutative polynomial ring in $n$ variables over $K$.
For any ideal $J \subset P_{K, n}$ we will denote the set of $K$-points cut out by $J$ in the affine space
$\mathbb A^n$ by 
\[ \mathcal{Z}(J) = \{ (a_1, \ldots, a_n) \in K^n \, | \, f(a_1, \ldots, a_n) = 0 \; \, \forall \; f \in J \} . \]
Conversely, for any set $X$ of $K$-points in $\bbA^n$, we will denote the ideal of polynomials vanishing at every point of $X$
by $\mathcal{I}(X) \subset P_{K, n}$. If $K$ is an algebraically closed field, then the celebrated Hilbert Nullstellensatz asserts that
\begin{equation} \label{null-alg-closed} \mathcal{I}(\mathcal{Z}(J)) = \Rad(J), \end{equation}
where $\Rad(J) = \{ f \in P_{K, n} \, | \, f^m \in J \; \, \text{for some integer $m \geqslant 1$} \}$ is the radical of $J$.

If $K$ is not algebraically closed, then $\mathcal{I}(\mathcal{Z}(J))$ is larger than $\Rad(J)$ in general. However, several variants of the Nullstellensatz are known to hold in this setting. In particular, if $K = \bbR$ is the field of real numbers, then for every ideal $J \subset \mathbb R[x_1, \ldots, x_n]$,
we have
\begin{equation} \label{e.real-null} \mathcal{I}(\mathcal{Z}(J)) = \Rad_{\mathbb R}(J) , \end{equation}
where 
\[
\begin{split}
\Rad_{\mathbb R}(J) = \{ f \in P_{K, n} \; | \; & \exists \; m,r \geqslant 1 \; \text{and} \;  
f_1, \ldots, f_{m-1} \in P_{K, n} \\   
                        & \text{such that $f_1^2 + \ldots + f_{m-1}^2 + f^{2r} \in J$} \}.  
\end{split}
\]
This assertion is known as the ``real Nullstellensatz"; see~\cite[Corollary 4.2.5]{prestel}. The ideal $\Rad_{\mathbb R}(J)$ 
is called the real radical of $J$. For the history of 
the real Nullstellensatz and further references, we refer the reader to~\cite[Section 4.7]{prestel}. 

There is also a version of the Nullstellensatz for a general field $K$. To state it, recall that a homogeneous polynomial
$p \in K[z_1, \ldots, z_m]$ is called anisotropic if $p(b_1, \ldots, b_m) = 0$ for some $b_1, \ldots, b_m \in K$ is
only possible when $b_1 = \ldots = b_m = 0$. Similarly we will say that a homogeneous polynomial $p \in K[z_1, \ldots, z_m]$  
is quasi-anisotropic if $p(b_1, \ldots, b_m) = 0$ for some $b_1, \ldots, b_m \in K$ is only possible when $b_m = 0$. Then
\begin{equation} \label{e.K-null}
\mathcal{I}(\mathcal{Z}(J)) = \Rad_K(J), 
\end{equation}
where
\begin{align*}
\Rad_K(J)  = & \; \{ f \in P_{K, n} \; | \; \exists \; \text{$m \geqslant 1$, a homogeneous quasi-anisotropic 
                $p \in K[z_1, \ldots, z_m]$}  \\
                        &  \; \text{and $f_1, \ldots, f_{m-1} \in P_{K, n}$  such that $p(f_1, \ldots, f_{m-1}, f) \in J$} \}.   
\end{align*}
This identity, known as the $K$-Nullstellensatz, is due to Laksov~\cite[p.~324]{laksov}. In weaker form it
appeared earlier in the work of Adkins-Gianni-Tognoli~\cite{agt}; see Section~\ref{sect.D-radical}. 
The ideal $\Rad_K(J) \subset K[x_1, \ldots, x_n]$ 
is called the $K$-radical of $J$. Laksov~\cite{laksov} showed that \eqref{e.K-null} can fail if ``quasi-anisotropic" 
is replaced by ``anisotropic" in the definition of $\Rad_K(J)$.

The purpose of this note is to explore a non-commutative version of the Nullstellensatz, following up on 
a recent paper by Alon and Paran~\cite{alon-paran}. In the sequel $K$ will denote an infinite field and $D$ 
a finite-dimensional central division algebra of degree $s$ over $K$. The role of the polynomial ring $P_{K, n}$ 
will be played by the ring $P_{D, n}$ of polynomial functions $f \colon D^n \to D$.
Here we refer to a function $f(x_1, \ldots, x_n) \colon D^n \to D$ as a polynomial function if 

\smallskip
(1) $f$ is expressible as a sum of finitely many ``monomial" functions of the form
\[ (x_1, \ldots, x_n) \mapsto a_1 x_{i_1} a_2 x_{i_2} \cdot \ldots \cdot a_r x_{i_r}  a_{r+1} , \]
for some $a_1, \ldots, a_{r+1} \in D$.

\smallskip
Alternatively, if $b_1, \ldots, b_{s^2}$ is a $K$-basis for $D$, then we can identify $D$ with $K^{s^2}$
and $D^n$ with $K^{ns^2}$. Writing $x_i = y_{i1} b_1 + \ldots + y_{is^2} b_{s^2}$ for each $i = 1, \ldots, n$ and
expanding $f(x_1, \ldots, x_n)$, we obtain
\[ f(x_1, \ldots, x_n) = f_1(y_{ij}) b_1 + \ldots + f_{s^2}(y_{ij}) b_{s^2},  \]
where each $f_i$ is a function $K^{ns^2} \to K$.
We can now define $f(x_1, \ldots, x_n) \colon D^n \to D$ to be a polynomial function if 

\smallskip
(2) each $f_k(y_{ij})$ is a polynomial in the variables $y_{ij}$ ($i = 1, \ldots, n$ and $j = 1, \ldots, s^2$)
with coefficients in $K$.

\smallskip
\noindent
It is easy to see that (2) does not depend on the choice of $K$-basis $b_1, \ldots, b_{s^2}$ in $D$. 
Moreover, (1) and (2) are equivalent; see Wilczy\'{n}ski~\cite[Theorem 4.1]{wil} or Alon-Paran~\cite[Theorem 5]{alon-paran}.

As in the commutative case, for any two-sided ideal $J \subset P_{D,n}$, we can define the zero locus of $J$ in $D^n$ as 
\[ \mathcal{Z}(J) = \{ (d_1, \ldots, d_n) \subset D^n \, | \, f(d_1, \ldots, d_n) = 0 \; \forall f \in J \}. \] 
Similarly for any subset $X \subset D^n$, we can define the two-sided ideal 
\[ \mathcal{I}(X) = \{ f \in P_{D,n} \; | \; f(d_1, \ldots, d_n) = 0 \; \forall \; (d_1, \ldots, d_n) \in X  \} \]
of $P_{D,n}$.

In the case, where $K = \mathbb R$ is the field of real numbers, and $D = \mathbb H$, the algebra of Hamilton quaternions, Alon and Paran
proved the following ``quaternionic" version of the Nullstellensatz:
\begin{equation} \label{e.H-null}
\begin{split}
\mathcal{I}(\mathcal{Z}(J)) = &  \; \{ f \in P_{\mathbb H, n} \; | \;  \exists \; m , r\geqslant 1 \; \text{and} \;   f_1, \ldots, f_{m-1} \in P_{\mathbb H, n} \\
 & \; \text{such that
   $N_{\mathbb H}(f_1) +  \ldots + N_{\mathbb H}(f_{m-1}) + N_{\mathbb H}(f)^r  \in J \}$;}
\end{split} \end{equation}
see~\cite[Theorem 1]{alon-paran}.
Here $N_{\mathbb H}(f)$ denotes the composition of $f \colon \mathbb{H}^n \to \mathbb{H}$ with the norm function 
$N_{\mathbb H} \colon \mathbb H \to \mathbb R \hookrightarrow \mathbb H$. Their proof is based on the real Nullstellensatz~\eqref{e.real-null}.

Our goal in this note is to extend~\eqref{e.H-null} to two-sided ideals 
$J$ in $P_{D, n}$, for an arbitrary infinite field $K$ and arbitrary central division $K$-algebra $D$. 
To state our main result, we need to introduce some additional notations.

We will denote the center of $P_{D, n}$ by $C_{D, n}$. It is easy to see that $C_{D, n}$ is the ring of polynomial maps $D^n \to K$, 
where we identify $K$ with the center of $D$. In other words, $C_{D, n}$ is the commutative polynomial 
ring $K[y_{ij}]$ in $ns^2$ variables, where $i = 1, \ldots, n$ and $j = 1, \ldots, s^2$. Note that in the case, where $D = K$ is a field 
 (i.e., $s = 1$), $P_{D, n} = C_{D, n} = P_{K, n} = K[x_1, \ldots, x_n]$. 

There are two natural candidates for the radical of a $2$-sided ideal $J \subset P_{D, n}$ suggested by~\eqref{e.H-null}, $\Rad_D(J)$ and $\Rad_D'(J)$ defined below. Here $J_c = J \cap C_{D, n}$.
\begin{align*} 
\Rad_D(J) = & \; \{ f \in P_{D, n} \, | \, N_D(f) \in \Rad_K(J_c) \}  \\
           = &  \; \{ f \in P_{D, n} \; | \; \exists \; \text{$m \geqslant 1$, a homogeneous quasi-anisotropic 
                $p \in K[z_1, \ldots, z_m]$}  \\
                        &  \; \text{and $f_1, \ldots, f_{m-1} \in C_{D, n}$  such that $p(f_1, \ldots, f_{m-1}, N_D(f)) \in J \}$.} 
\end{align*}
and
\begin{align*}  \Rad'_D(J) =  & \;   \{ f \in P_{D, n} \; | \; \exists \; \text{$m \geqslant 1$, a homogeneous quasi-anisotropic 
                $p \in K[z_1, \ldots, z_m]$}  \\
                        & \;  \text{and $f_1, \ldots, f_{m-1} \in P_{D, n}$  such that $p(N_D(f_1), \ldots, N_D(f_{m-1}), N_D(f)) \in J \}$.} 
\end{align*}

Once again, by the norm $N_{D}(f)$ of $f \in P_{D, n}$ we mean the composition of $f \colon D^n \to D$ 
with the norm map $N_{D} \colon D \to K$, where we identify $K$ with the center of $D$.
We are now ready to state the main result of this paper.

\begin{thm} \label{thm.main}  Let $K$ be an infinite field,
$D$ be a finite-dimensional central division algebra of degree $s$ over $K$, 
and $J$ be a $2$-sided ideal in $P_{D, n}$. Then

\smallskip
(a) $\mathcal{I}(\mathcal{Z}(J)) = \Rad_D(J)$.

\smallskip
(b) Moreover, if $\operatorname{char}(K)$ does not divide $s!$, then $\Rad_D(J) = \Rad'_D(J)$.
\end{thm}

In particular, Theorem~\ref{thm.main} tells us that $\Rad_D(J)$ is always a 2-sided ideal in $P_{D, n}$,
which is not obvious from the definition. Similarly Theorem~\ref{thm.main}(b) tells us that $\Rad'_D(J)$ 
is a $2$-sided ideal of $P_{D, n}$, provided that $\operatorname{char}(K)$ does not divide $s!$.

When $D = K$ (i.e., $s = 1$), then $P_{D,n} = C_{D,n} = K[x_1, \ldots, x_n]$ is the usual commutative polynomial ring,
$N_D(f) = f$, and Theorem~\ref{thm.main} reduces to the $K$-Nullstellensatz~\eqref{e.K-null}.
Note however that we will not give a new proof in this case. We will assume the $K$-Nullstellensatz 
and deduce Theorem~\ref{thm.main} from it.

The assumption that $K$ is infinite is made to ensure that formal polynomials are in bijective correspondence
with polynomial maps $D^n \to D$. This assumption is harmless: if $K$ is finite, then by Wedderburn's Little Theorem, 
the only finite-dimensional central division algebra over $K$ is $D = K$ itself. This places us back in the setting of the commutative 
$K$-Nullstellensatz~\eqref{e.K-null}. Thus we are not missing anything by assuming that $K$ is infinite.

The main weakness of the $K$-Nullstellensatz~\eqref{e.K-null} is that for most fields $K$ 
the set of quasi-anisotropic polynomials is difficult to describe explicitly. 
Theorem~\ref{thm.main} tells us that, as in the quaternionic case~\eqref{e.H-null}, 
no additional complications arise in passing from $K$ to $D$; see also Proposition~\ref{prop.D-radical}.

\section{Proof of Theorem~\ref{thm.main}(a)}

Let $J$ be a $2$-sided ideal of $P_{D, n}$. Set $J_c = J \cap C_{D, n}$, as we did in the Introduction.
By \cite[Proposition 17.5]{gw}, $J$ is generated by $J_c$ as an ideal. In particular,
\begin{equation} \label{e.z} Z_D(J) = Z_K(J_c).
\end{equation}
Here $Z_D(J)$ is a subset of $D^n$ and $Z_K(J_c)$ is a subset of $K^{ns^2}$, and we are using a $K$-basis
$b_1, \ldots, b_{s^2}$ to identify $D^n$ with $K^{ns^2}$.

Suppose $f \in \Rad_D(J)$, i.e., $N_{D}(f) \in \Rad_K(J_c)$.  Then there exists a positive integer $m$, a homogeneous
quasi-anisotropic polynomial $p \in K[z_1, \ldots, z_m]$ and $f_1, \dots, f_{m-1} \in J_c$ 
such that $p(f_1, \dots, f_{m-1}, N_D(f)) \in J_c$. Hence, for any $a \in \mathcal{Z}(J_c) = \mathcal{Z}(J)$, 
we have $p(f_1(a), \dots, f_{m-1}(a), N_D(f)(a)) =0$. Since $p$ is quasi-anisotropic, this implies $N_D(f)(a)=0$.
Since $D$ is a central division algebra over $K$, this is only possible if $f(a)=0$. We conclude that    
$f(a)=0$ for any $a \in \mathcal{Z}(J)$. So $f \in \mathcal{I}(\mathcal{Z}(J))$. This shows that    
$\Rad_D(J) \subset \mathcal{I}(\mathcal{Z}(J))$.
    
To prove the opposite inclusion, let $f \in \mathcal{I}(\mathcal{Z}(J))$.  Write $f = f_1 b_1 + \ldots + f_{s^2} b_{s^2}$, where each
$f_i$ lies in $C_{D, n}$. To say that $f$ vanishes at every point of $\mathcal{Z}(J)$ is equivalent 
to saying that each $f_i$ vanishes at 
every point of $\mathcal{Z}(J)$. In view of~\eqref{e.z}, the latter is equivalent to each $f_i$ vanishing 
at every point of $\mathcal{Z}(J_c)$. In other words, $f_i \in \mathcal{I}(\mathcal{Z}(J_c))$ for every $i = 1, \ldots, s^2$.
By the $K$-Nullstellensatz~\eqref{e.K-null}, $\mathcal{I}(\mathcal{Z}(J_c))=\Rad_K(J_c)$, so each $f_i$ lies in $\Rad_K(J_c)$.
Since $N_D(f)$ is a homogeneous polynomial of degree $s$ in $f_1, \dots, f_{s^2}$ with 
coefficients in $K$, we conclude $N_D(f) \in \Rad_K(J_c)$. By the definition of $\Rad_D(J)$, this is equivalent to
$f \in \Rad_D(J)$, as desired. \qed

\begin{example} The quaternionic Nullstellensatz~\eqref{e.H-null} is readily recovered from Theorem~\ref{thm.main}(a).
Indeed, let $K = \mathbb R$ be the field of real numbers and $D = \mathbb H$ be the algebra of Hamilton quaternions. 
Following Along and Paran~\cite{alon-paran}, we define the quaternionic radical of an ideal $J \subset P_{\bbH , n}$ as
\begin{align*} 
\Rad_{AP}(J) = &  \{ f \in P_{\bbH , n} \; | \;  \exists \; m, k \geqslant 1,  \; \text{and} \;
  f_1, \ldots, f_{m-1} \in P_{\bbH , n} \; \text{such that} \\ 
& N_{\mathbb H}(f_1) + \ldots + N_{\mathbb H}(f_{m-1}) + N_{\mathbb H}(f)^k \in J \}. \end{align*}

To recover the quaternionic Nullstellensatz~\eqref{e.H-null} from Theorem~\ref{thm.main}, it suffices to show that 
\[ \Rad_{\mathbb H}(J) \subset \Rad_{AP}(J) \subset \mathcal{I}( \mathcal{Z}(J) ). \]
To prove that $\Rad_{AP}(J) \subset  \mathcal{I}(\mathcal{Z}(J))$, suppose that $f \in \Rad_{AP}(J)$. 
Then there exists $m,k \geqslant 1$, and $f_1, \dots, f_{m-1} \in P_{\bbH , n}$, such 
that $N_{\bbH} (f_1) + \dots + N_{\bbH}(f_{m-1}) + N_{\bbH} (f)^k \in J$. 
So for any $a \in \mathcal{Z}(J)$, we have $(N_{\bbH} (f_1) + \dots + N_{\bbH}(f_{m-1}) + N_{\bbH} (f)^k) (a)=0$.
Since $N_{\bbH}(b)$ is non-negative for any $b \in \mathbb H$, this implies $f(a)=0$. In other words,
$f$ vanishes at every point of $\mathcal{Z}(J)$, i.e., $f \in \mathcal{I}(\mathcal{Z}(J))$, as desired.

To prove that $\Rad_{\mathbb H}(J) \subset \Rad_{AP}(J)$, suppose that $f \in \Rad_{\bbH} (J) $. 
Then by our definition of $\Rad_{\mathbb H}(J)$, we have $N_{\mathbb H}(f) \in \Rad_{\mathbb R}(J_c)$. 
By the Real Nullstellensatz~\eqref{e.real-null}, 
there exists $f_1, \dots, f_{m-1} \in C_{D, n}$ and $m \geqslant 1 $, such that 
$f_1^2 + \dots + f_{m-1}^2 +N_{\bbH}(f)^{2k} \in J_c$ for some $k \geq 1$. 
Since $f_i \in C_{D, n}$, $f_i^2=N_{\bbH}(f_i)$. In other words,
$N_{\mathbb H}(f_1) + \ldots + N_{\mathbb H}(f_{m-1}) + N_{\mathbb H}(f)^{2k} \in J$. Thus 
$f \in \Rad_{AP}(J)$, as desired.
\qed
\end{example}

\section{Proof of Theorem~\ref{thm.main}(b)}

It is clear from the definition that $\Rad'_D(J) \subset \Rad_D(J)$. Our proof of the opposite inclusion will rely on the following elementary proposition.

\begin{prop} \label{prop.polynomial} 
Let $K$ be a field and $s \geqslant 1$ be an integer. If $\operatorname{char}(K)$ does not divide $s!$, 
then the $s$-powers span the polynomial ring $K[y_1, \ldots, y_m]$ as a $K$-vector space. In other words, any polynomial 
$f \in K[y_1, \dots, y_m]$, can be written as a $K$-linear combination of
polynomials of the form $g^s$, where $g \in K[y_1, \ldots, y_m]$.
\end{prop}

\begin{proof} 

We claim that the following identity holds in any commutative ring $R$:
\begin{equation} \label{e.polarization}
t_1 \ldots t_s = \frac{1}{s!} \sum_{I \subset \{  1, \ldots, s \}} (-1)^{s - |I|} \, (\sum_{i \in I} t_i)^s \, . \end{equation}
Here $t_i \in R$ 
and $|I|$ denotes the number of elements in a subset $I$ of $\{ 1, \dots s \}$. Once~\eqref{e.polarization} is established, 
we can substitute $R = K[y_1, \ldots, y_m]$, $t_1 = f$ and $t_2 = \ldots = t_s = 1$ to obtain a desired expression for $f$ 
as a linear combination of $s$-powers in $K[y_1, \ldots, y_m]$.

It thus remains to prove~\eqref{e.polarization}. This identity is a special case of the ``polarization formula", 
relating a homogeneous polynomial of degree $s$ to its associated $s$-linear form; see~\cite{se}. 
For the sake of completeness, we include a short self-contained proof 
of~\eqref{e.polarization}~\footnote{A slightly modified version of this same argument can be used to prove the general polarization formula.}.

The right hand side of~\eqref{e.polarization} is a homogeneous polynomial of degree $s$. To determine the coefficient of 
a monomial $M$ in $t_1, \ldots, t_s$ of degree $s$, let us consider two cases.

\smallskip
Case 1: One of the variables $t_1, \ldots, t_s$ does not occur in $M$.
For notational simplicity, let us say this variable is $t_s$.
To calculate the coefficient of $M$ on the right hand side of~\eqref{e.polarization}, 
we may set $t_s$ equal to $0$. 
The index sets $I \subset \{ 1, \ldots, s \}$ not containing $s$ are in bijective correspondence with index sets $I' = I \cup \{ s \}$ 
containing $s$. Since the size of $I'$ is one greater that the size of $I$ the terms $(\sum_{i \in I} t_i)^s$ and $(\sum_{i \in I' } t_i)^s$
appear with opposite signs in the outer sum. When we substitute $t_s = 0$, these terms will cancel, and the right hand side will sum up to $0$ .
This shows that the coefficient of $M$ on the right hand side of~\eqref{e.polarization} is $0$.

\smallskip
Case 2. The only monomial $M$ of degree $s$ not covered by Case 1 is $M = t_1 \dots t_s$. This monomial can only come 
from the $(t_1+\dots +t_s)^s$ term on the right hand side of~\eqref{e.polarization}, and it comes with coefficient $s!$. 

\smallskip
This completes the proof of~\eqref{e.polarization} and thus of Proposition~\ref{prop.polynomial}.
\end{proof}

We are now ready to finish the proof of Theorem~\ref{thm.main}(b) by showing that \[ \Rad_D(J) \subset \Rad'_D(J). \]
Let $f \in \Rad_D(J)$. Our goal is to show that $f \in \Rad'_D(J)$. By the definition of $\Rad_D(J)$, 
there exists a homogeneous quasi-anisotropic $p \in K[z_1, \dots, z_m]$ 
and $f_1, \dots, f_{m-1} \in C_{D,n}$, such that $p(f_1, \dots, f_{m-1}, N_D(f)) \in J$.

As we explained in the Introduction, after choosing a $K$-basis $b_1, \ldots, b_{s^2}$ in $D$,
we can identify $C_{D, n}$ with the polynomial ring $K[y_{ij}]$
in $n s^2$ variables.
By Proposition~\ref{prop.polynomial}, there exist finitely many polynomials
$g_1, \dots, g_l \in C_{D,n}$, such that each $f_i$ can be written as 
$$f_i=c_{i1} g_1^s + \dots + c_{il} g_l^s$$
for some $c_{ij} \in K$. (The proof of Proposition~\ref{prop.polynomial} shows that
we can take each $c_{ij}$ to be either $0$ or $\pm \dfrac{1}{s!}$, but this will
not matter in the sequel.)
Let $$q(w_1, \dots, w_l, w_{l+1})= p(z_1, \dots, z_{m-1}, y_m), $$ where 
$ z_i=c_{i1}w_1 + \dots + c_{il} w_l$ for each  $i=1, \dots, m-1$
and $z_m= w_{l+1}$. 
This way we obtain a new homogeneous polynomial $q(w_1, \dots, w_l, w_{l+1})$ in $l+1$ variables. 
Since $p$ is quasi-anisotropic, it is clear from the definition that $q$ is also quasi-anisotropic and 
$$q(g_1^s, \dots, g_l^s,N_D(f))=p(f_1, \dots, f_{m-1},N_D(f)) \in J.$$
Since $g_i \in C_{D,n}$, we have $N_D(g_i) = g_i^s$ and 
thus \[ q(N_D(g_1), \dots, N_D(g_l), N_D(f)) = q(g_1^s, \dots, g_l^s, N_D(f)) \in J. \] 
This shows that $f \in \Rad'_D(J)$, as desired.
\qed

\section{$D$-radical ideals}
\label{sect.D-radical}

An ideal $J$ in $P_{K, n} = K[x_1, \ldots, x_n]$ is called $K$-radical if $J = \mathcal{Z}(\mathcal{I}(J))$. The following description of $K$-radical ideals is due to Adkins, Gianni and Tognoli~\cite{agt}: $J$ is $K$-radical if and only $\Rad_K(J) = J$. This assertion is 
a precursor to (and, in turn, an easy consequence of) the $K$-Nullstellensatz~\eqref{e.K-null}, due to Laksov. 

Now let $D$ be a finite-dimensional division algebra over $K$. We will say that a two-sided ideal $J \subset P_{D, n}$ is a $D$-radical if
$J = \mathcal{I}(\mathcal{Z}(J))$. In view of Theorem~\ref{thm.main} this is equivalent to $J = \Rad_D(J)$. 
We conclude this paper with the following observation.

\begin{prop} \label{prop.D-radical} A two-sided ideal $J \subset P_{D, n}$ is $D$-radical if and only if $J_c = J \cap C_{D, n}$ is $K$-radical. 
\end{prop}

\begin{proof} Recall that $\mathcal{Z}(J) = \mathcal{Z}(J_c)$ for any $2$-sided ideal $J \subset P_{D, n}$; see~\eqref{e.z}.

First suppose $J \subset P_{D, n}$ is $D$-radical. Clearly $J_c \subset \mathcal{I}(\mathcal{Z}(J_c))$. To prove the opposite inclusion,
assume that $f \in C_{D, n}$ lies in $\mathcal{I}(\mathcal{Z}(J_c))$. Since $J$ is $D$-radical, we conclude that $f$ lies in $J$ and hence, 
in $J \cap C_{D, n} = J_c$, as desired.

Conversely, suppose $J_c$ is $K$-radical. Once again, $J \subset \mathcal{I}(\mathcal{Z}(J))$ for any two-sided ideal $J \subset P_{D, n}$, so we only need to prove the opposite inclusion. Assume that $f \in \mathcal{I}(\mathcal{Z}(J)) \subset P_{D, n}$. 
Choose a $K$-basis $b_1, \ldots, b_{s^2}$ of $D$ and write $f = f_1 b_1 + \ldots + f_{s^2} b_{s^2}$, where each $f_i$ lies in $C_{D, n}$. 
Since $f$ vanishes on $\mathcal{Z}(J) = \mathcal{Z}(J_c)$, so does each $f_i$. Since $J_c$ is $K$-radical, 
we conclude that each $f_i$ lies in $J_c$ and hence, $f = f_1 b_1 + \ldots + f_{s^2} b_{s^2}$ lies in $J$. 
\end{proof}

\end{document}